\documentclass[12 pt,twoside]{amsart}

\usepackage{amsopn}
\usepackage{amssymb}
\usepackage{amscd}

\newtheorem{theorem}{Theorem}[section]
\newtheorem{lemma}[theorem]{Lemma}
\newtheorem{proposition}[theorem]{Proposition}

\theoremstyle{definition}

\theoremstyle{remark}
\newtheorem{remark}[theorem]{Remark}

\numberwithin{equation}{section}

\begin{document}

\title[Isometries of a locally compact space with one end]{The group of isometries of a locally compact metric space with one end}

\author[Antonios Manoussos]{Antonios Manoussos}
\address{Fakult\"{a}t f\"{u}r Mathematik, SFB 701, Universit\"{a}t Bielefeld, Postfach 100131, D-33501 Bielefeld, Germany}
\email{amanouss@math.uni-bielefeld.de}
\thanks{During this research the author was fully supported by SFB 701 ``Spektrale Strukturen und
Topologische Methoden in der Mathematik" at the University of Bielefeld, Germany. He would also like to express
his gratitude to Professor H. Abels for his support.}


\date{}


\begin{abstract}
In this note we study the dynamics of the natural evaluation action of the group of isometries $G$ of a locally compact metric space $(X,d)$ with one end. Using the
notion of pseudo-components introduced by S. Gao and A. S. Kechris we show that $X$ has only finitely many pseudo-components exactly one of which is not compact and
$G$ acts properly on this pseudo-component. The complement of the non-compact component is a compact subset of $X$ and $G$ may fail to act properly on it.
\end{abstract}

\maketitle

\section{Preliminaries and the main result}
The idea to study the dynamics of the natural evaluation action of the group of isometries $G$ of a locally compact metric space $(X,d)$ with one end, using the
notion of pseudo-components introduced by S. Gao and A. S. Kechris in \cite{kechris}, came from a paper of E. Michael \cite{michael1}. In this paper he introduced the
notion of a $J$-space, i.e. a topological space with the property that whenever $\{ A,B\}$ is a closed cover of $X$ with $A\cap B$ compact, then $A$ or $B$ is
compact. In terms of compactifications locally compact non-compact $J$-spaces are characterized by the property that their end-point compactification coincides with
their one-point compactification (see \cite[Proposition 6.2]{michael1}, \cite[Theorem 6]{nowi}). Recall that the Freudenthal or end-point compactification of a
locally compact non-compact space $X$ is the maximal zero-dimensional compactification $\varepsilon X$ of $X$. By zero-dimensional compactification of $X$ we here
mean a compactification $Y$ of $X$ such that $Y\setminus X$ has a base of closed-open sets (see \cite{macar}, \cite{nowi}). The points of $\varepsilon X\setminus X$
are the ends of $X$. From the topological point of view locally compact spaces with one end are something very general since the product of two non-compact locally
compact connected spaces is a space with one end (see \cite[Proposition 8]{nowi}, \cite[Proposition 2.5]{michael1}), so it is rather surprising that the dynamics of
the action of the group of isometries $G$ of a locally compact metric space $(X,d)$ with one end has a certain structure as our main result shows.

\begin{theorem}\label{main}
Let $(X,d)$ be a locally compact metric space with one end and let $G$ be its group of isometries. Then

\begin{enumerate}
\item[(i)] $X$ has finitely many pseudo-components exactly one of which  is not compact and $G$ is locally compact.

\item[(ii)] Let $P$ be the non-compact pseudo-component. Then $G$ acts properly on $P$, $X\setminus P$ is a compact subset of $X$ and $G$ may fail to act properly on it.
\end{enumerate}
\end{theorem}

Let us now recall some basic notions. Let $(X,d)$ be a locally compact metric space let $G$ be its group of isometries. If we endow $G$ with the topology of pointwise
convergence then $G$ is a topological group (see \cite[Ch. X, \S 3.5 Corollary]{bour2}). On $G$ there is also the topology of uniform convergence on compact subsets
which is the same as the compact--open topology. In the case of a group of isometries these topologies coincide with the topology of pointwise convergence, and the
natural evaluation action of $G$ on $X$, $G\times X\to X$ with $(g,x)\mapsto g(x)$, $g\in G$, $x\in X$ is continuous (see \cite[Ch. X, \S 2.4 Theorem 1 and \S 3.4
Corollary 1]{bour2}). An action by isometries is proper if and only if the limit sets $L(x)=\{ y\in X\,|\, \mbox{there exists a net}\, \{g_{i}\}\, \mbox{in}\,\,
G\,\,\mbox{with}\,\, g_{i}\to\infty\,\, \mbox{and}\,\, \lim g_{i}x=y\}$ are empty for every $x\in X$, where $g_i\to\infty$ means that the net $\{ g_i\}$ has no
cluster point in $G$ (see \cite{clopen}). A few words about pseudo-components. They were introduced by S. Gao and A. S. Kechris in \cite{kechris} and we used them in
\cite{clopen} to study the dynamics of the action of the group of isometries of a locally compact metric space. For the convenience of the reader we repeat what a
pseudo-component is. For each point $x\in X$ we define the radius of compactness $\rho (x)$ of $x$ as $\rho (x):=\sup \{\, r>0\, |\,\,B(x,r)\,\,\mbox{ has compact
closure}\}$ where $B(x,r)$ denotes the open ball centered at $x\in X$ with radius $r>0$. If $\rho (x)=+\infty$ for some $x\in X$ then every ball has compact closure
(i.e. $X$ has the Heine-Borel property), hence $\rho (x)=+\infty$ for every $x\in X$. In case where $\rho (x)$ is finite for some $x\in X$ then the radius of
compactness is a Lipschitz function \cite[Proposition 5.1]{kechris}. It is also easy to see that $\rho (gx)=\rho (x)$ for every $g\in G$. We define next an
equivalence relation $\mathcal{E}$ on $X$ as follows: Firstly we define a directed graph $\mathcal{R}$ on $X$ by $x\mathcal{R} y$ if and only if $d(x,y)<\rho (x)$.
Let $\mathcal{R}^*$ be the transitive closure of $\mathcal{R}$, i.e. $x\mathcal{R}^* y$ if and only if for some $u_0=x, u_1, \ldots, u_n=y$ we have $u_i\mathcal{R}
u_{i+1}$ for every $i<n$. Finally, we define the following equivalence relation $\mathcal{E}$ on $X$: $x\mathcal{E} y$ if and only if $x=y$ or $(x\mathcal{R}^* y$ and
$y\mathcal{R}^* x)$. We call the $\mathcal{E}$-equivalence class of $x\in X$ the pseudo-component of $x$, and we denote it by $C_x$. It follows that pseudo-components
are closed-open subsets of $X$, see \cite[Proposition 5.3]{kechris} and $gC_x=C_{gx}$ for every $g\in G$.

Before we give the proof of Theorem \ref{main} we need some results that may be of independent interest.

\begin{lemma} \label{lemClopPart}
Let $X$ be a non-compact $J$-space and let $\mathcal{A}=\{ A_i,\, i\in I\}$ be a partition of $X$ with closed-open non-empty sets. Then $\mathcal{A}$ contains only
finitely many sets exactly one of which is not compact; its complement is a compact subset of $X$.
\end{lemma}
\begin{proof}
We show firstly that there exists a set in $\mathcal{A}$ which is not compact. We argue by contradiction. Assume that every set $B\in \mathcal{A}$ is compact. Then
$\mathcal{A}$ contains infinitely many distinct sets because otherwise $X$ must be a compact space. Let $\{ B_n,\, n\in\mathbb{N}\}\subset \mathcal{A}$ with $B_n\neq
B_k$ for $n\neq k$ (i.e. $B_n\cap B_k=\emptyset$). The sets $\displaystyle{D=:\bigcup_{n=1}^{+\infty} B_{2n-1}}$ and $X\setminus D$ are open (since $X\setminus D$ is
a union of elements of $\mathcal{A}$) and disjoint so they form a closed partition of $X$. Hence, one of them must be compact. This is a contradiction because both
$D$ and $X\setminus D$ are an infinite disjoint union of open sets.

Fix a non-compact $P\in \mathcal{A}$. Since $P$ is a closed-open subset of $X$ then $\{ P, X\setminus P\}$ is a closed partition of $X$. Hence $P$ or $X\setminus P$
must be compact. But $P$ is non-compact so $X\setminus P$ is compact. If $K\in \mathcal{A}$ with $K\neq P$ then $K\subset X\setminus P$. Therefore, $K$ is compact.
Moreover $\mathcal{A}$ contains finitely many sets, since $X\setminus P$ is compact and $\mathcal{A}$ is a partition of $X$ with closed-open non-empty sets.
\end{proof}

The previous lemma makes $X$ a second countable space (i.e. $X$ has a countable base):

\begin{proposition} \label{propSep}
A metrizable locally compact $J$-space has a countable base.
\end{proposition}
\begin{proof}
Sierpinski has proved in \cite{sierp} that every metrizable locally separable space X can be represented as a disjoint union of open separable subsets. Then Lemma
\ref{lemClopPart} implies that we have here only finitely many of these sets, and hence, $X$ is second countable.
\end{proof}

The proof of Theorem \ref{main} is heavily based on the next proposition. Its proof can be found in \cite[Theorem 1.3]{clopen} but we repeat it here for the
convenience of the reader.

\begin{proposition}\label{pseudo}
Let $(X,d)$ be a locally compact metric space and let $G$ denote its group of isometries. Let $x,y\in X$ and a net $\{g_i\}$ in $G$ such that $g_ix\to y$. Then there
exist a subnet $\{g_j\}$ of $\{g_i\}$ and a map $f:C_x\to X$ which preserves the distance such that $g_{j}\to f$ pointwise on $C_x$, $f(x)=y$ and $f(C_x)=C_{f(x)}$,
where $C_x$ and $C_y$ denote the pseudo-components of $x$ and $y$ respectively. In the case where $X$ has, moreover, a countable base and $\{ g_i\}$ is a sequence,
then there exist a subsequence $\{ g_{i_k}\}$ of $\{ g_i\}$ and a map $f:C_x\to X$ which preserves the distance such that $g_{i_k}\to f$ pointwise on $C_x$, $f(x)=y$
and $f(C_x)=C_{f(x)}$.

\end{proposition}
\begin{proof}
Let $F$ be a subset of $G$. We define $K(F)$ to be the set
\[
K(F):=\{ x\in X\, |\,\,\mbox{the set}\,\, Fx\,\,\mbox{has compact closure in}\,\, X\}.
\]
Each $K(F)$ is a closed-open subset of $X$ (see \cite[Lemma 3.1]{manstra}, \cite{str1}).

Let $x, y\in X$ and $\{ g_i\}$ be a net in $G$ with $g_ix\to y$.  Since $X$ is locally compact there exists an index $i_0$ such that the set $F(x)$, where $F:=\{
g_i\,|\,i\geq i_0\}$, has compact closure. We claim that for every $z\in C_x$ the set $F(z)$ has, also, compact closure in $X$, hence $C_x\subset K(F)$. The strategy
is to start with an open ball $B(x,r)$ centered at $x$ with radius $r<\rho (x)$, where $\rho (x)$ is the radius of compactness of $x$ and prove that $F(z)$ has
compact closure for every $z\in B(x,r)$. Then, our claim follows just from the definition of the pseudo-component of $x$. To prove the claim take a sequence $\{g_n
z\}$, with $g_n\in F$ for every $n\in\mathbb{N}$. Since the closure of $F(x)$ is compact we may assume, without loss of generality, that $g_nx\to w$ for some $w$ in
the closure of $F(x)$. Assume that $\rho (x)$ is finite and take a positive number $\varepsilon$ such that $r+\varepsilon < \rho (x)$. Then for $n$ big enough
\[
d(g_nz,w)\leq d(g_nz,g_nx) + d(g_nx,w)=d(z,x)+ d(g_nx,w)<r+\varepsilon < \rho (x).
\]
Recall that the radius of convergence is a continuous map, and since $g_nx\to w$ then $ \rho (x)= \rho (w)$. So, the sequence $\{g_n z\}$ is contained eventually in a
ball of $w$ with compact closure, hence it has a convergence subsequence. The same also holds in the case where $\rho (x)=+\infty$ and the claim is proved.

Set $A:=K(F)$. By \cite[Lemma 3.1]{manstra} $A$ is a closed-open subset of $X$. If $g_i|_A$ denotes the restriction of each $g_i$ on $A$, then the Arzela-Ascoli
theorem implies that the set $\{ g_i|_A:A\to X\, |\, i\geq i_0 \}$ has compact closure in $C(A,X)$ (this the set of all continuous maps from $A$ to $X$). Thus, there
exist a subnet $\{g_j\}$ of $\{g_i\}$ and a map $f:A\to X$ with $f(x)=y$ which preserves the distance such that $g_j\to f$ pointwise on $A$. Hence, $g_{j}\to f$
pointwise on $C_x$. If, moreover, $X$ has a countable base then it is $\sigma$-compact, i.e. it can be written as a countable union of compact subsets. Since $A=K(F)$
is a closed-open subset of $X$ then it also a $\sigma$-compact locally compact metrizable space. Hence, by \cite[Theorem 5.2 p. 265 and 8.5 p. 272]{doug}, $C(A,X)$ is
a metrizable space with a countable base. So, if $\{ g_i\}$ is a sequence there exist a subsequence $\{ g_{i_k}\}$ of $\{ g_i\}$ and a map $f:C_x\to X$ which
preserves the distance such that $g_{i_k}\to f$ pointwise on $C_x$.

Let us show that $f(C_x)=C_{f(x)}$. Since $d(x,g_j^{-1}f(x))=d(g_jx,f(x))$ and $d(g_jx,f(x))\to 0$ it follows that $g_j^{-1}f(x)\to x$. Repeating the previous
procedure, there exist a subnet $\{g_k\}$ of $\{g_j\}$ and a map $h:C_{f(x)}\to X$ which preserves the distance such that $g_{k}^{-1}\to h$ pointwise on $C_{f(x)}$
and $h(f(x))=x$. Note that since $g_{k}x\to f(x)$ and the pseudo-component $C_{f(x)}$ is a closed-open subset of $X$ then $g_{k}x \in C_{f(x)}$ eventually for every
$k$. Therefore, $g_kC_x=C_{g_{k}x}=C_{f(x)}$. Take a point $z\in C_x$. Then $g_{k}z\to f(z)$ and since the pseudo-component $C_{f(x)}$ is a closed-open subset of $X$
then $f(z)\in C_{f(x)}$, so $f(C_x)\subset C_{f(x)}$. In a similar way and repeating the same arguments as before it follows that $hC_{f(x)}\subset C_x$. Take now a
point $w\in C_{f(x)}$. Then $h(w)\in C_x$, hence $g_{k}^{-1}(w)\in C_x$ eventually for every $k$. So, $w=g_{k}g_{k}^{-1}(w)\to f(h(w))\in f(C_x)$ from which follows
that $C_{f(x)}\subset f(C_x)$.
\end{proof}

\begin{proof}[Proof of Theorem \ref{main}]
(i) Since every pseudo-component is a closed-open subset of $X$ we can apply Lemma \ref{lemClopPart} for the family of the pseudo-components of $X$. Hence, $X$ has
finitely many pseudo-components exactly one of which, say $P$, is not compact and its complement $X\setminus P$ is a compact subset of $X$. Take any $g\in G$. Then
$gP$ is a non-compact pseudo-component hence $gP=P$. This shows that $P$ is $G$-invariant. Then, by \cite[Corollary 6.2]{kechris}, $G$ is locally compact since $X$
has finitely many pseudo-components.

(ii) We shall show that $G$ acts properly on $P$. By Proposition \ref{propSep}, the space $(X,d)$ has a countable base. Hence, as we mentioned in the proof of
Proposition \ref{pseudo}, by \cite[Theorem 5.2 p. 265 and 8.5 p. 272]{doug}, $G$ is a metrizable locally compact group with a countable base. So, if we would like to
check if $G$ acts properly on $P$ it is enough to consider sequences in $G$ instead of nets. Assume that there are points $x,y\in P$ and a sequence $\{ g_n\}$ in $G$
with $g_nx\to y$. Let us denote by $\{ P, C_1, C_2, \ldots, C_k\}$ the pseudo-components of $X$. Each pseudo-component $C_i$, $i=1,\ldots,k$ is compact. Choose points
$x_i\in C_i$, $i=1,\ldots,k$. Since $X\setminus P$ is compact we may assume that there exist points $y_i\in X\setminus P$, $i=1,\ldots,k$ and a subsequence $\{
g_{n_{l}}\}$ of $\{ g_n\}$ such that $g_{n_{l}}x_i\to y_i$ for every $i=1,\ldots,k$. Since by Proposition \ref{propSep}, $X$ has a countable base then, by Proposition
\ref{pseudo}, there is a subsequence of $\{ g_{n_{m}}\}$ of $\{ g_n\}$ and a map $f:X\to X$ which preserves the distance such that $g_{n_{m}}\to f$ pointwise on $X$.
Note that $g_{n_{m}}^{-1}y\to x\in P$, since $d(g_{n_{m}}^{-1}y,x)=d(y,g_{n_{m}}x)$. Repeating the previous arguments we conclude that there exists a map $h:X\to X$
and a subsequence $\{ g_{n_{m_p}}\}$ of $\{ g_{n_{m}}\}$ such that $g_{n_{m_p}}^{-1}\to h$ pointwise on $X$ and $h$ preserves the distance. Obviously $h$ is the
inverse map of $f$, hence $f\in G$ and $G$ acts properly on $P$. The group $G$ may fail to act properly on $X\setminus P$. As an example we may take as $X=P\cup S
\subset\mathbb{R}^3$, where $P$ is the plane $\{ (x,y,0))\, |\, x,y\in\mathbb{R}\}$ and $S$ is the circle $\{ (x,y,2)\, |\, x^2+y^2=1\}$. We endow $X$ with the metric
$d=\min \{ d_E,1\}$, where $d_E$ is the usual Euclidean metric on $\mathbb{R}^3$. Then the action of $G$ on $S$ is not proper, since for a point $x\in S$ the isotropy
group $G_x:=\{ g\in G\,|\, gx=x\}$ is not compact.
\end{proof}

\begin{remark}
If $G$ does not act properly on $X\setminus P$ one may ask if the orbits on $X\setminus P$ are closed or if the isotropy groups of points $x\in X\setminus P$ are
non-compact. The answer is negative in general. As an example we may consider the example in \cite{abman1}.  In this paper we constructed a one-dimensional manifold
with two connected components, one compact isometric to $S^1$, and one non-compact, the real line with a locally Euclidean metric. It has a complete metric whose
group of isometries has non-closed dense orbits on the compact component. We can regard the real line as a distorted helix with a locally Euclidean metric. The
problem is that this manifold has two ends. But this is not really a problem. Following the same arguments as in \cite{abman1} we can replace the distorted helix by a
small distorted helix-like stripe and have a space with one end and two connected components, one compact isometric to $S^1$, and one non-compact with a locally
Euclidean metric so that the group of isometries has non-closed dense orbits on the compact component.
\end{remark}

\noindent \textbf{Acknowledgements.} We would like to thank the referee for an extremely careful reading of the manuscript and her/his valuable remarks and comments.

\end{document}